\documentclass[10pt]{article}
\usepackage{amsmath,amsthm,amssymb,amsfonts,amscd}
\usepackage{xcolor}
\newtheorem{theorem}{Theorem}[section]
\newtheorem{definition}[theorem]{Definition}
\newtheorem{lemma}[theorem]{Lemma}
\newtheorem{corollary}[theorem]{Corollary}

\newtheorem{example}[theorem]{Example}
\newtheorem{remark}[theorem]{Remark}

\begin{document}


\title 
{\textbf{Algebraic structure  of the $L_2$ analytic Fourier--Feynman transform 
associated with Gaussian processes on Wiener space}}


\author{{\sc Seung Jun Chang}\\
Department of Mathematics,  Dankook University\\
Cheonan 330-714, Korea\\
sejchang@dankook.ac.kr 
\and 
{\sc Jae Gil Choi\thanks{¢ÓCorresponding author.}}\\
Department of Mathematics, Dankook University\\
Cheonan 330-714,  Korea\\
jgchoi@dankook.ac.kr
\and
{\sc David Skoug}\\
Department of Mathematics, University of Nebraska-Lincoln\\
Lincoln, NE 68588-0130, USA\\ 
dskoug@math.unl.edu 
}

\date{\empty}

\maketitle

\begin* 

\vspace{-.5cm}
\noindent
\thanks{{\bf Abstract}:  
In this paper we study  algebraic structures
of the classes of the $L_2$ analytic Fourier--Feynman transforms
on Wiener space. To do this we first  develop  several rotation 
properties of the generalized Wiener integral associated with 
Gaussian processes. We then proceed to analyze the $L_2$ analytic 
Fourier--Feynman transforms  associated with 
Gaussian processes. Our results show that these $L_2$  
analytic  Fourier--Feynman transforms are actually linear operator 
isomorphisms from a Hilbert space into itself. We finally   investigate  the 
algebraic structures of these classes of the transforms on Wiener 
space, and show that they  indeed are group isomorphic.}\\

\addtolength{\columnsep}{15mm} 

\noindent
\thanks{{\bf Keywords}: 
Paley--Wiener--Zygmund stochastic integral,
Gaussian process,
Fourier--Feynman transform associated with Gaussian paths,
monoid isomorphism,
linear operator isomorphism,
free group.}\\

\noindent

\noindent
\thanks{{\bf 2010 Mathematics Subject Classification}: 
Primary 28C20, 60J65;  Secondary  60G15, 54H15}\\

\end*

\maketitle


\setcounter{equation}{0}

\setcounter{equation}{0}
\section{Introduction}\label{sec:intro}

\par
Let $C_0[0,T]$ denote one parameter Wiener space. 
Bearman's rotation theorem \cite{Bearman} for Wiener measure has played 
an important role in various research areas in mathematics and physics 
involving Wiener integration theory. Bearman's theorem was further 
developed by Cameron and Storvick \cite{CS76-2} and by Johnson and 
Skoug \cite{JS79-2} in their studies of Wiener integral equations.

\par
The concept of the generalized Wiener integral 
and the generalized  Feynman integral 
on $C_0[0,T]$ were introduced by Park and Skoug in \cite{PS91}, further studied by Chung,
Park and Skoug in \cite{CPS93}, extended by Park and Skoug in \cite{PS95}, 
and further studied by Huffman, Park and Skoug in  \cite{HPS97}.
In \cite{CPS93,HPS97,PS91,PS95}, 
the generalized Wiener integral was defined by the Wiener integral
\begin{equation}\label{eq:idea01}
\int_{C_0[0,T]}F (\mathcal{Z}_h(x,\cdot) )dm(x),
\end{equation}
where $\mathcal{Z}_h(x,\cdot)$ is the Gaussian path given by the stochastic
integral $\mathcal{Z}_h(x,t)=\int_0^th(s)dx(s)$ with $h\in L_2[0,T]$.

\par
On the other hand,
the concept of the analytic Fourier--Feynman transform (abbr. FFT) on the Wiener 
space $C_0[0,T]$, initiated by Brue \cite{Brue}, has been developed in 
the literature. This transform  and its properties are similar in many 
respects to the ordinary Fourier  transform of functions on an Euclidean space. 
For an elementary introduction to the analytic FFT, see \cite{SS04} 
and the references cited therein.
In particular, in \cite{HPS97}, Huffman, Park 
and Skoug introduced a generalized FFT 
associated with the Gaussian paths $\mathcal Z_h(x,\cdot)$. 
Since then, the generalized FFT was further developed by many
mathematicians. For instance, see \cite{CCC17,CSC12}.

\par  
In this paper we study  algebraic structures
of the classes of the $L_2$ analytic FFTs
on Wiener space. To do this we first develop  the Bearman-type  theorems for the generalized 
Wiener integral given by \eqref{eq:idea01}. Using these results,  we  then take a 
closer look at the $L_2$ analytic FFT  associated with the 
Gaussian paths $\mathcal Z_h(x,\cdot)$ (abbr. $\mathcal Z_h$-FFT) introduced by Huffman, 
Park, and Skoug in \cite{HPS97}.  Our results  indicate that  the $L_2$ analytic $\mathcal Z_h$-FFTs 
are linear operator isomorphisms from a Hilbert space of cylinder functionals on Wiener space 
into itself. Furthermore  the algebraic structures of these generalized transforms are examined. 
Based on this examination we know that these  classes of generalized 
transforms are indeed group isomorphic.

\setcounter{equation}{0}
\section{Preliminaries}\label{pre}
 
\par
Given a positive real $T>0$, let  $C_0[0,T]$ denote one-parameter Wiener space, 
that is, the space of  all  real-valued continuous functions $x$ on the compact 
interval $[0,T]$  with $x(0)=0$. Let $\mathcal{M}$ denote the class of  all 
Wiener measurable subsets of $C_0[0,T]$ and let $m_w$ denote Wiener measure 
which is a Gaussian measure on $C_0[0,T]$ with mean zero and covariance function 
$r(s,t)=\min\{s,t\}$.  Then, as is well-known, $(C_0[0,T],\mathcal{M},m_w)$ is a 
complete  measure space.

\par 
A subset $B$ of $C_0[0,T]$ is said to be scale-invariant measurable \cite{JS79-2}  
provided $\rho B\in \mathcal{M}$ for all $\rho>0$, and a scale-invariant measurable 
set $N$ is said to be scale-invariant null provided  $m(\rho N)=0$ for all $\rho>0$.
A property that holds except on  a scale-invariant  null set is said to be hold  
scale-invariant almost everywhere (abbr. s-a.e.). A functional $F$ is said to be 
scale-invariant measurable provided $F$ is defined on a scale-invariant measurable 
set and $F(\rho\,\,\cdot\,)$ is Wiener  measurable for every $\rho>0$. 
If two functionals $F$ and $G$ are equal s-a.e.,  we write $F\approx G$.  
 
\par
The Paley--Wiener--Zygmund (abbr. PWZ) stochastic integral \cite{PWZ33} plays 
a key role throughout this paper. For  $v$ in $L_2[0,T]$, the PWZ stochastic 
integral $\langle{v,x}\rangle$  is given by  the formula 
\[
\langle{v,x}\rangle
:=\lim\limits_{n\to \infty} 
\int_0^T\sum\limits_{j=1}^n(v,\phi_j)_2\phi_j(t)d x(t)
\]
for $m$-a.e. $x\in C_0[0,T]$, where $\{\phi_n\}$ is a complete orthonormal 
set of functions of bounded variation on $[0,T]$ and $(\cdot,\cdot)_2$ denotes 
the $L_2$-inner product. 

\par
It is  known  that for each $v\in L_2[0,T]$, the  PWZ integral $\langle{v,x}\rangle$  
is essentially independent of the choice of  the complete  orthonormal set  $\{\phi_n\}$.
If $v$ is of bounded variation on $[0,T]$  then $\langle{v,x}\rangle$ equals the 
Riemann--Stieltjes integral $\int_0^T v(t)dx(t)$  for s-a.e. $x\in C_0 [0,T]$,  
and   for all $v\in L_2[0,T]$,  $\langle{v,x}\rangle$ is a Gaussian random variable 
on $C_0[0,T]$ with mean zero and variance $\|v\|_2^2$. For a more  detailed study of 
the PWZ stochastic integral, see \cite{JS81, PS88}.
  
\par
Given a function $h$ in $L_2[0,T]$ with $\|h\|_2>0$, let  $\mathcal{Z}_h(x,t)$ 
be the PWZ stochastic integral 
\begin{equation}\label{eq:g-process}
\mathcal{Z}_h(x,t)
:=\langle{h\chi_{[0,t]},x}\rangle
\end{equation}
where $\chi_{[0,t]}$ denotes the indicator function of the set $[0,t]$. Next, let
\begin{equation}\label{var-beta} 
\beta_h(t):=\int_0^t h^2(u)du. 
\end{equation}
Then the stochastic process  $\mathcal{Z}_h$ on $C_0[0,T]\times[0,T]$,
$(x,t)\mapsto \mathcal Z_h (x,y)$, is a 
Gaussian process with mean zero and covariance function
\[
\int_{C_0[0,T]}
\mathcal{Z}_h(x,s)\mathcal{Z}_h(x,t)dm(x)
=\beta_h({\min \{s,t\}}) .
\]
In addition, by \cite[Theorem 21.1]{Yeh73}, $\mathcal{Z}_h (\cdot, t)$ 
is stochastically  continuous in $t$ on $[0,T]$.
If $h \in L_2[0,T]$ is of bounded variation on $[0,T]$, then for all $x\in C_0[0,T]$,
$\mathcal Z_h(x,t)$ is continuous in $t$. Also, for any  
$h_1,h_2 \in  L_2[0,T]$,
\[
\int_{C_0[0,T]}\mathcal{Z}_{h_1}(x,s)\mathcal{Z}_{h_2}(x,t)dm(x)
=\int_{0}^{\min\{s,t\}}h_1(u)h_2(u) d u.
\]
Of course if $h(t)\equiv 1$ on $[0,T]$, 
then the process $\mathcal W$ on $C_0[0,T]\times[0,T]$ given by  
$(w,t)\mapsto\mathcal W_t(x)= \mathcal  Z_1 (x,t)=x(t)$
is a Wiener process.  
We note that the coordinate  process $\mathcal Z_1$ is stationary in time, 
whereas the stochastic process  $\mathcal{Z}_h$ generally is not.
For more detailed studies on the stochastic process  $\mathcal{Z}_h$,
see \cite{CPS93,PS91}.

From \cite[Lemma 1]{CPS93}, it follows  that for each $v\in L_2[0,T]$ and  
$h\in L_{\infty}[0,T]$,  
\begin{equation}\label{eq:pwz-property} 
\langle{v,\mathcal Z_h(x,\cdot)}\rangle
=\langle{vh,x}\rangle
\end{equation}
for  s-a.e. $x\in C_0[0,T]$. 

\par
Throughout the rest of this paper let  $\mathbb C_+$ and $\widetilde{\mathbb C}_+$ 
denote the set of complex numbers with positive real part and nonzero complex 
numbers  with nonnegative real part, respectively.

\par
Let $h$ be a function in $L_2[0,T]$ with $\|h\|_2>0$
and let  $F$ be a complex-valued scale-invariant measurable functional on $C_0[0,T]$ 
such that 
\[
J(h;\lambda):=\int_{C_0[0,T]}F(\lambda^{-1/2}\mathcal Z_h(x,\cdot))dm(x)
\]  
exists  as a finite number for all $\lambda>0$. If there exists a function 
$J^* (h;\lambda)$ analytic on $\mathbb C_+$  such that  $J^*(h;\lambda)=J(h;\lambda)$ 
for all $\lambda > 0$, then  $J^* (h;\lambda)$ is defined to be the  
analytic  Wiener integral (associated with Gaussian paths $\mathcal Z(x,\cdot)$)  
of $F$ over $C_0[0,T]$ with parameter $\lambda$, 
and for  $\lambda \in \mathbb C_+$ we write 
\[
\int_{C_0[0,T]}^{\text{\rm anw}_{\lambda}} 
F (\mathcal{Z}_h(x,\cdot) )dm(x):= J^*(h;\lambda).
\]
Let $q$ be a nonzero real number and let $F$ be a functional such that
\[
\int_{C_0[0,T]}^{\text{\rm anw}_{\lambda}}
F (\mathcal{Z}_h(x,\cdot) )dm(x)
\]
exists for all $\lambda \in \mathbb C_+$.
If the following limit exists, we call it the   analytic  Feynman  
integral (associated with Gaussian paths $\mathcal Z_h(x,\cdot)$) of $F$ with parameter $q$, and we write
\[
I_h^{\text{\rm anf}_{q}}[F]  
\equiv I_{h,x}^{\text{\rm anf}_{q}}[F(\mathcal Z_h(x,\cdot))] 
:= \lim_{\lambda \to -iq} 
\int_{C_0[0,T]}^{\text{\rm anw}_{\lambda}} 
F(\mathcal{Z}_h(x,\cdot))dm(x) 
\] 
where $\lambda$ approaches $-iq$ through values in $\mathbb C_+$.

\par
We are now ready to state the definition of the $L_2$ analytic 
$\mathcal Z_h$-FFT  on Wiener space.

\begin{definition} \label{def:fft}
Let $h$ be a function in $L_2[0,T]$ with $\|h\|_2>0$ and 
let $F$ be a scale-invariant measurable functional on $C_0[0,T]$.
 For $\lambda \in \mathbb{C}_+$ 
and $y \in C_{0}[0,T]$, let 
\[
T_{\lambda,h}(F)(y)
:=\int_{C_0[0,T]}^{\text{\rm anw}_{\lambda}} 
F (y+\mathcal{Z}_h(x,\cdot) )dm(x). 
\]
Also, let $q$ be a nonzero real number. We define the   $L_2$  analytic $\mathcal{Z}_h$-FFT,
$T^{(2)}_{q,h}(F)$ of $F$, 
by the formula,
\[
T^{(2)}_{q,h}(F)(y)
:= \operatorname*{l.i.m.}_{\begin{subarray}{1}
\lambda\to -iq \\  \,\,\lambda\in \mathbb C_+\end{subarray}} 
T_{\lambda,h} (F)(y)    				 
\]
if it exists; i.e.,  for each $\rho>0$,
\[
\lim_{\begin{subarray}{1}
\lambda\to -iq \\  \,\,\lambda\in \mathbb C_+\end{subarray}} 
\int_{C_{0}[0,T]} 
\big| T_{\lambda,h} (F)(\rho y)
   -T^{(2)}_{q, h }(F)(\rho y) \big|^{2} dm(y)=0 .
\]
\end{definition}

\par
We note that  $T_{q,h}^{(2)}(F)$ is  defined only s-a.e.. We also note that 
if $T_{q,h}^{(2)}(F)$ exists  and if $F\approx G$, then $T_{q,h}^{(2)}(G)$ exists
and  $T_{q,h}^{(2)}(G)\approx T_{q,h}^{(2)}(F)$. One can also see that for each 
$h\in L_2[0,T]$, $T_{q,h}^{(2)}(F)\approx T_{q,-h}^{(2)}(F)$ since
\begin{equation}\label{symmetric}
\int_{C_0[0,T]}F(x)dm(x)=\int_{C_0[0,T]}F(-x)dm(x).
\end{equation}

\begin{remark}  \label{referee-01}
Frankly speaking, it follows that
for any functions   $h_1$ and $h_2$  in $L_2[0,T]$ satisfying 
the equality  $h_1^2=h_2^2$  $\mathrm{m}_L$-a.e. on $[0,T]$
 ($\mathrm{m}_L$ denote the Lebesgue measure on $[0,T]$ in this paper),
\[ 
\int_{C_0[0,T]}F(\mathcal Z_{h_1}(x,\cdot))dm(x)
=\int_{C_0[0,T]}F(\mathcal Z_{h_2}(x,\cdot))dm(x) 
\] 
in the sense that if either side exists, both sides exist and equality holds.
Thus, for any scale-invariant measurable functional $F$ on $C_0[0,T]$, 
\begin{equation}\label{eq:referee-01}
 T_{q,h_1}^{(2)}(F)\approx T_{q,h_2}^{(2)}(F), 
\end{equation}
because the Gaussian processes  $\mathcal Z_{h_1}$
and $\mathcal Z_{h_2}$  have the same distribution.
\end{remark}

\begin{remark}  
Note that if $h\equiv 1$ on $[0,T]$, then the definition of the $L_2$ 
analytic  $\mathcal Z_1$-FFT agrees with the previous definition of the  $L_2$
analytic Fourier--Feynman transform \cite{CS76-1,HPS95,JS79-1}. 
\end{remark}

\setcounter{equation}{0}
\section{Rotation of Wiener measures  associated with\\ Gaussian paths}

\par
In this section  we  develop  rotation theorems for the generalized Wiener integral
of the functionals on $C_0[0,T]$. 
Throughout this section we will assume that each functional $F:C_0[0,T]\to\mathbb R$ we consider 
is scale-invariant measurable  and that
\[
\int_{C_{0}[0,T]}\big|F(\rho \mathcal Z_h(x,\cdot))\big|dm(x)<+\infty
\]
for each $\rho>0$ and for each  $h\in L_2[0,T]$.

\par
Let $v_1$ and $v_2$ 
be functions in $L_{2}[0,T]$ with  $\|v_1\|_2^2=\|v_2\|_2^2\equiv \sigma^2>0$. 
The  random variables $X_1(x)=\langle{v_1,x}\rangle$  and  
$X_2(x)=\langle{v_2,x}\rangle$ will then have the same distribution, $N(0,\sigma^2)$.

\par  
Next let $h_1$ and $h_2$ be functions in $L_2[0,T]$. 
Then there exists a function  $\mathbf{s}\in L_2[0,T]$ such that
\begin{equation}\label{eq:fn-rot} 
\mathbf{s}^2(t)=h_1^2(t)+h_2^2(t) 
\end{equation} 
for $\mathrm{m}_L$-a.e. $t\in [0,T]$. We note that for any nonzero function $\mathbf{s}_1$ 
(``nonzero function $g$'' means that $\|g\|_2>0$ throughout this paper) 
in $L_2[0,T]$, 
there exists  $\mathbf{s}_2\in L_2[0,T]$  such that $\|\mathbf{s}_1\|_2=\|\mathbf{s}_2\|_2$, 
but, $\mathbf{s}_1 \ne \mathbf{s}_2$   $\mathrm{m}_L$-a.e. on $[0,T]$  
(more generally, $\mathrm{m}_L(\{t\in [0,T]: \mathbf{s}_1(t) \ne \mathbf{s}_2(t)\})>0$). 
Thus one can see that the 
function `$\mathbf{s}$' satisfying \eqref{eq:fn-rot} is not unique. Let $\mathbf{s}_1$ 
and $\mathbf{s}_2$ be functions in $L_{2}[0,T]$ that satisfy equation \eqref{eq:fn-rot}. 
Then, in view of the   observation above,  
$\langle{\mathbf{s}_1, x}\rangle$ and $\langle{\mathbf{s}_2, x}\rangle$ have the same 
distribution $N(0,\|h_1\|_2^2+\|h_2\|_2^2)$.

 \begin{remark}
Consider the   relation $\sim$ on
$L_2[0,T]$ given by 
\begin{equation}\label{relation-basic}
h\sim g \,\,\Longleftrightarrow\,\, h^2=g^2
\,\,\,\mathrm{m}_L\mbox{-a.e.}.
\end{equation}
Then $\sim$ is an equivalence relation.
The equivalence relation  $\sim$ is also well-defined on $L_{\infty}[0,T]$.
\end{remark}

We will use the symbol  
$\mathbf{s}(h_1,h_2)$ for the   functions `$\mathbf{s}$' that satisfy
\eqref{eq:fn-rot} above. 
Given functions  $h_1$ and $h_2$  in $L_{2}[0,T]$ (resp. $L_{\infty}[0,T]$),
infinitely many functions, $\mathbf{s}(h_1,h_2)$, exist   in $L_{2}[0,T]$ 
(resp. $L_{\infty}[0,T]$).
 Thus the $\mathbf{s}(h_1,h_2)$'s  can be considered as equivalence classes
of the   relation \eqref{relation-basic}, i.e.,  it allow  us that
\[
\mathbf{s}(h_1,h_2)=\Big[\sqrt{h_1^2+h_2^2}\Big]
:=\Big\{\mathbf{s}:\mathbf{s}\sim \sqrt{h_1^2+h_2^2}\Big\}.
\]  
But, by the observation above, it follows that for every function $\mathbf{s}$ 
in the equivalence class $\mathbf{s}(h_1,h_2)$, the Gaussian random variable 
$\langle{\mathbf{s},x}\rangle$ has the normal distribution 
$N(0,\|h_1\|_2^2+\|h_2\|_2^2)$.

\par
We also note that  if the functions $h_1$ and $h_2$ are in  $L_{\infty}[0,T]$, 
then we can take $\mathbf{s}(h_1,h_2)$ to 
be in $L_{\infty}[0,T]$).

\par
Inductively, given a finite sequence
$\mathcal H=(h_1,\ldots, h_n)$ of functions in $L_2[0,T]$ 
(resp. $L_{\infty}[0,T]$), let  $\mathbf{s}(h_1,h_2,\ldots,h_n)$
be the functions $\mathbf{s}$ in $L_2[0,T]$ 
(resp. $L_{\infty}[0,T]$) which satisfy the relation 
\begin{equation}\label{eq:s-def}
\mathbf{s}^2(t)=h_1^2(t)+\cdots+h_n^2(t) 
\end{equation}
for $m_L$-a.e. $t\in [0,T]$. 
Then for $k\in\{1,\ldots, n-1\}$,
\[
\begin{aligned}
\mathbf{s}^2(\mathbf{s}(h_1,\ldots,h_k),h_{k+1})(t)
&=\mathbf{s}^2(h_1,\ldots,h_k)(t)+h_{k+1}^2(t)  =\sum_{j=1}^{k+1}h_j^2(t) \\
&=\mathbf{s}^2(h_1,\ldots,h_k ,h_{k+1})(t)
\end{aligned}
\]
for $\mathrm{m}_L$-a.e. $t\in [0,T]$.
For convenience, given a finite sequence
$\mathcal H=(h_1,\ldots, h_n)$ we denote $\mathbf{s}(h_1,\ldots,h_n)$ 
by  $\mathbf{s}(\mathcal H)$. 
From these arguments,  we have  the following properties:

\par
(i)  For any finite sequence $\mathcal H=(h_1,\ldots,h_n)$  in
$L_2[0,T]$,
$\mathbf{s}(\mathcal H)$ has an consistency property. That is, for any  
permutation $\pi$ of $I_n=\{1,2,\ldots,n\}$,
\begin{equation}\label{eq:pre-eq}
 \mathbf{s}(h_1,h_2,\ldots,h_n)
=\mathbf{s}(h_{\pi(1)},h_{\pi(2)}\ldots,h_{\pi(n)})  
\end{equation}
for $\mathrm{m}_L$-a.e. $t\in [0,T]$.  

\par
(ii) Given two sequences $\mathcal H_1=(h_{11},h_{12},\ldots,h_{1n_1})$ and
$\mathcal H_2=(h_{21},h_{22},\ldots,h_{2n_2})$ of functions in 
 $L_2[0,T]$,
let $\mathcal H_1 \wedge \mathcal H_2$ denote the cementation of 
$\mathcal H_1$ and $\mathcal H_2$, i.e., 
\begin{equation}\label{wedge-H}
\mathcal H_1 \wedge \mathcal H_2
:=(h_{11},h_{12},\ldots,h_{1n_1},h_{21},h_{22},\ldots,h_{2n_2}).
\end{equation} 
Then 
\begin{equation}\label{s-wedge-H} 
\mathbf{s}(\mathcal H_1 \wedge \mathcal H_2)
=\mathbf{s}(\mathbf{s}(\mathcal H_1),\mathbf{s}(\mathcal H_2)). 
\end{equation} 

\par
For $h_1,h_2\in L_2[0,T]$ with $\|h_j\|_{2}>0$, $j\in\{1,2\}$, 
let $\mathcal Z_{h_1}$ and $\mathcal Z_{h_2}$ be the Gaussian processes given 
by \eqref{eq:g-process} with $h$ replaced with $h_1$ and $h_2$ respectively. 
Then the process
\[
\mathfrak Z_{h_1,h_2}: C_{0}[0,T]\times C_{0}[0,T]\times [0,T]\to \mathbb R
\]
given by
\[
\mathfrak Z_{h_1,h_2}(x_1,x_2,t)
:= \mathcal Z_{h_1}(x_1,t)+ \mathcal Z_{h_2}(x_2,t)
\]
is also a Gaussian process with mean zero and covariance 
\[
\begin{aligned}
&\int_{C_{0}^2[0,T]}
\mathfrak Z_{h_1,h_2}(x_1,x_2,s)
\mathfrak Z_{h_1,h_2}(x_1,x_2,t)d m^2(x_1,x_2)\\
&=\beta_{h_1}(\min\{s,t\})+\beta_{h_2}(\min\{s,t\})
\end{aligned}
\] 
where $\beta_h$ is given by \eqref{var-beta} above.

\par
Next we consider the stochastic process 
$\mathcal Z_{\mathbf{s}(h_1,h_2)}: C_0[0,T]\times[0,T]\to\mathbb R$. 
As stated in Section \ref{pre} above, the process $\mathcal Z_{\mathbf{s}(h_1,h_2)}$ 
is Gaussian with mean zero and covariance
\[
\begin{aligned}
\beta_{\mathbf{s}(h_1,h_2)}(\min\{s,t\})
&=\int_0^{\min\{s,t\}} \mathbf{s}^2(h_1,h_2)(u)  du\\
&=\int_0^{\min\{s,t\}}\big(h_1^2(u)+h_2^2(u)\big) du\\
&= \beta_{h_1}(\min\{s,t\})+\beta_{h_2}(\min\{s,t\}).
\end{aligned}
\]

\par
From these facts, one can see that  $\mathfrak Z_{h_1,h_2}$ 
and $\mathcal Z_{\mathbf{s}(h_1,h_2)}$ have the same distribution.
Thus we obtain the following theorems and corollaries.

\begin{theorem}
Let $F$ be a functional on $C_0[0,T]$, and let $h_1$ and $h_2$ be  nonzero functions 
in $L_2[0,T]$. Then
\begin{equation}\label{R001}
\begin{aligned}
&\int_{C_{0}^2[0,T]}F(\mathcal Z_{h_1}(x_1,\cdot)+\mathcal Z_{h_2}(x_2,\cdot))dm^2(x_1,x_2)\\
&=\int_{C_{0}[0,T]}F(\mathcal Z_{\mathbf{s}(h_1,h_2)}(x,\cdot))dm(x).
\end{aligned}
\end{equation}
\end{theorem}

\begin{remark}
In \cite{CC13}, the authors established the rotation theorem, namely, equation 
\eqref{R001}, only in the case that the integrand functionals $F$ are cylinder type on $C_0[0,T]$. 
But the  integrand functionals $F$ in \eqref{R001} are not restricted. 
\end{remark}

\begin{corollary}
Let $F$ be a functional on $C_0[0,T]$, and let $\mathcal H=(h_{1},h_{2},\ldots,h_{n})$ 
be a finite  sequence of nonzero functions in $L_2[0,T]$. Then
\[
\int_{C_{0}^n[0,T]}F\Bigg(\sum_{j=1}^n\mathcal Z_{h_j}(x_j,\cdot)\Bigg)dm^n(\vec x)
=\int_{C_{0}[0,T]}F(\mathcal Z_{\mathbf{s}(h_1,\ldots,h_n)}(x,\cdot))dm(x).
\]
\end{corollary}

\begin{theorem}
Let $F$ be a functional on $C_0[0,T]$, and let $\mathcal H_{1}$
and $\mathcal H_{2}$ be finite sequences  of nonzero  functions in  $L_2[0,T]$.
Then
\begin{equation}\label{R002}
\begin{aligned}
&\int_{C_{0}^2[0,T]}F(\mathcal Z_{\mathbf{s}(\mathcal H_1)}(x_1,\cdot)
+\mathcal Z_{\mathbf{s}(\mathcal H_2)}(x_2,\cdot))dm^2(x_1,x_2)\\
&=\int_{C_{0}[0,T]}F(\mathcal Z_{\mathbf{s}(\mathbf{s}(\mathcal H_1),\mathbf{s}(\mathcal H_2))}
(x,\cdot) )dm(x)\\
&=\int_{C_{0}[0,T]}F(\mathcal Z_{\mathbf{s}(\mathcal H_1\wedge \mathcal H_2)}(x,\cdot))dm(x).
\end{aligned}
\end{equation}
\end{theorem}

\begin{corollary}
Let $F$ be a functional on $C_0[0,T]$,  and 
let $\mathcal H=(h_{1},h_{2},\ldots,h_{n})$ be a finite sequence of nonzero  functions in $L_2[0,T]$
and let $h_{n+1}$ be a nonzero  function in  $L_2[0,T]$. Then
\begin{equation}\label{R007}
\begin{aligned}
&\int_{C_{0}^2[0,T]}F(\mathcal Z_{\mathbf{s}(\mathcal H)}(x_1,\cdot)
+\mathcal Z_{h_{n+1}}(x_2,\cdot))dm^2(x_1,x_2)\\
&=\int_{C_{0}[0,T]}F(\mathcal Z_{\mathbf{s}(\mathbf{s}(\mathcal H),h_{n+1})}(x,\cdot))dm(x)\\
&=\int_{C_{0}[0,T]}F(\mathcal Z_{\mathbf{s}(\mathcal H \wedge (h_{n+1}))}(x,\cdot))dm(x).
\end{aligned}
\end{equation}
\end{corollary}

\begin{remark}
Using seminal results by Bearman \cite{Bearman}, Cameron and Storvick \cite{CS76-2} 
established  a rotation property of the Wiener  measure $m$.  The result is summarized 
as follows: for a Wiener-integrable functional  $F$ and  every nonzero real  $a$ and $b$,
\begin{equation}\label{old}
\int_{C_{0}^2[0,T]}  F(ax+by)d(m\times m)(x,y) 
\stackrel{*}{=}\int_{C_{0}[0,T]}
F (\sqrt{a^2+b^2}z )dm(z), 
\end{equation}
where by $\stackrel{*}{=}$ we mean that if either side exists, both sides exist 
and equality holds.

\par
Let  $h_1\equiv a$ and $h_2\equiv b$ be constant functions on $[0,T]$ 
in  \eqref{R001} above. Then, 
it follows that
\[
\begin{aligned}
&\int_{C_{0}^2[0,T]}  F(ax+by)dm^2(x,y)\\
&=\int_{C_{0}^2[0,T]}  F(\mathcal Z_{a}(x,\cdot)+\mathcal Z_{b}(y,\cdot))dm^2(x,y)  \\
&=\int_{C_{0}[0,T]}F (\mathcal Z_{\mathbf{s}(a,b)}(z,\cdot) )dm(z). 
\end{aligned}
\]
Furthermore, we can choose $\mathbf{s}(a,b)$ as a constant function.
In this case, one can see that  either $\mathbf{s}(a,b)=\sqrt{a^2+b^2}$ 
or $\mathbf{s}(a,b)=-\sqrt{a^2+b^2}$.
Thus, in view of equation \eqref{symmetric},
we obtain equation \eqref{old} as a special case of  \eqref{R001}.
\end{remark}

\begin{remark}
For any finite sequence $\mathcal H=(h_1,\ldots,h_n)$ of nonzero functions  in $L_2[0,T]$, 
let $\mathbf{s}_1$ and  $\mathbf{s}_2$
be the   functions satisfying equation \eqref{eq:s-def}.
Then, in view of  equation \eqref{eq:referee-01}, it follows that
for a  functional  $F$ on $C_0[0,T]$, 
\[
T_{q,\mathbf{s}_1}^{(2)}(F)\approx T_{q,\mathbf{s}_2}^{(2)}(F).
\]
From this we see that 
$T_{q,\mathbf{s}(\mathcal H)}^{(2)}(F)$ is consistent with
the representation of the function $\mathbf{s}(\mathcal H)$.
\end{remark}

\begin{remark} As mentioned in above,
if the functions $h_1$ and $h_2$ are in $L_{\infty}[0,T]$, 
then we can take $\mathbf{s}(h_1,h_2)$ to be in $L_{\infty}[0,T]$.
Thus equations \eqref{R001} through \eqref{R007} 
hold for the Gaussian processes whose kernel is in $L_{\infty}[0,T]$. 
\end{remark}

\setcounter{equation}{0}
\section{$L_2$  analytic $\mathcal Z_h$-Fourier--Feynman transform}

\par
In this section, we  analyze  the $L_2$ analytic $\mathcal Z_h$-FFT of cylinder 
functionals.
A functional $F$ on $C_{0}[0,T]$ is called a cylinder functional if 
there exists a linearly independent subset $\{v_1,\ldots, v_m\}$
of functions in $L_2[0,T]$ such that
\begin{equation}\label{eq:cylinder1}
F(x)
=\psi(\langle{v_1,x}\rangle,\ldots,\langle{v_m,x}\rangle),
\quad x \in C_{0}[0,T],
\end{equation}
where $\psi$ is a complex-valued Lebesgue measurable function on $\mathbb R^m$.

\par
It is easy to show  that for the given cylinder functional $F$ of the
form \eqref{eq:cylinder1}, there exists an orthogonal  subset
$\mathcal A=\{\alpha_1,\ldots,\alpha_n\}$  of $L_2[0,T]$ such that
$F$ can be  expressed as
\begin{equation}\label{eq:cylinder-x}
F(x)
=f(\langle{\alpha_1,x}\rangle,\ldots,\langle{\alpha_n,x}\rangle),
\quad x \in C_{0}[0,T],
\end{equation}
where $f$ is a complex-valued Lebesgue measurable function on $\mathbb R^n$.
Thus, we loose no generality in assuming that every cylinder  functional on
$C_{0}[0,T]$ is of the form  \eqref{eq:cylinder-x}.

\par
For $h\in L_{\infty}[0,T]$ with $\|h\|_{2}>0$, let  $\mathcal{Z}_h$ be
the Gaussian process given by \eqref{eq:g-process} above and let $F$ be
given by equation \eqref{eq:cylinder-x}. Then  by equation \eqref{eq:pwz-property},
\[
\begin{aligned}
F(\mathcal{Z}_h (x,\cdot))
&=f(\langle{\alpha_1,\mathcal{Z}_h (x,\cdot)}\rangle,
    \ldots ,\langle{\alpha_n,\mathcal{Z}_h (x,\cdot)}\rangle)\\
&=f(\langle{\alpha_1h,x}\rangle, \ldots ,\langle{\alpha_n h,x}\rangle).
\end{aligned}
\]
Even though the  set   $\mathcal{A}=\{\alpha_1,\ldots, \alpha_n\}$
of functions in $L_2[0,T]$ is  orthogonal,  the subset 
$\mathcal{A}h\equiv\{\alpha_1h,\ldots,\alpha_nh\}$ of $L_2[0,T]$ need not  
be  orthogonal.

\par
Let $\mathcal{A}=\{\alpha_1,\ldots,\alpha_n\}$  be an orthogonal set of
nonzero functions in $L_2[0,T]$, and  let $\mathcal{O}_{\infty}(\mathcal{A})$
denote the class of all nonzero functions $h$ in $L_{\infty}[0,T]$ such that 
$\mathcal{A}h$ is orthogonal in $L_2[0,T]$. Since $\dim L_2[0,T]=\infty$,  
infinitely many functions, $h$, exist in $\mathcal{O}_{\infty}(\mathcal{A})$.

\begin{example}
For any orthogonal set $\mathcal{A}=\{\alpha_1,\ldots, \alpha_n\}$ of nonzero
functions in $L_2[0,T]$, each of whose element is of bounded variation
on $[0,T]$, let $L(S)$ be the subspace of $L_2[0,T]$ which is spanned by
$S=\left\{\alpha_i\alpha_j:1\leq i<j\leq n\right\}$, and let $L(S)^{\perp}$ 
be the orthogonal
complement of $L(S)$. Let
\[
\mathcal{P}_{\infty}(\mathcal{A})
:=\{ h\in L_{\infty}[0,T] : h^2 \in L(S)^{\perp} \hbox{ and }  \|h\|_{2}>0\}.
\]
Since $\dim L(S)$ is finite, and $L_{\infty}[0,T]$ is dense in $L_2[0,T]$,
$\dim(L(S)^{\perp}\cap L_{\infty}[0,T]) = \infty$ and 
so $\mathcal{P}_{\infty}(\mathcal{A})$ has infinitely many elements.

\par
Let $h$ be an element of $\mathcal{P}_{\infty}(\mathcal{A})$.
It is easy to show that  $\|\alpha_j h\|_{2}>0$ for all  $j\in\{1,\ldots,n\}$.
From the definition of the $\mathcal{P}_{\infty}(\mathcal{A})$, we see that
for $i,j\in\{1,\ldots,n\}$ with $i\ne j$,
\[
(\alpha_i h, \alpha_j h)_{2}
=\int_0^T \alpha_i(t) \alpha_j(t) h^2(t) dt=0.
\]
From these, we see that  $\mathcal{A} h$ is an orthogonal set in $L_2[0,T]$
for any $h$   in $\mathcal{P}_{\infty}(\mathcal{A})$, i.e., 
$\mathcal{P}_{\infty}(\mathcal{A}) \subset \mathcal{O}_{\infty}(\mathcal{A})$.
Other examples can be found in \cite{CC13}.
\end{example}

\par
Given  $h_1$ and $h_2$ in $\mathcal O_{\infty}(\mathcal A)$,
let $\mathbf{s}(h_1,h_2)$ be an element of $L_{\infty}[0,T]$ 
which satisfies equation \eqref{eq:fn-rot} above.
Then we observe that for all $j,l\in\{1,\ldots,n\}$ with $j\ne l$,
\begin{equation}\label{eq:check01}
\begin{aligned}
(\alpha_j \mathbf{s}(h_1,h_2), \alpha_l \mathbf{s}(h_1,h_2))_2 
&=\int_0^T \alpha_j(t)\alpha_l(t) \mathbf{s}^2(h_1,h_2)(t) dt\\
&=\int_0^T \alpha_j(t)\alpha_l(t) (h_1^2(t)+h_2^2(t)) dt\\
&=\int_0^T \alpha_j(t)\alpha_l(t)h_1^2(t)dt
+\int_0^T \alpha_j(t)\alpha_l(t)h_2^2(t) dt\\
&= (\alpha_j h_1, \alpha_lh_1)_2 +(\alpha_jh_2, \alpha_lh_2)_2   =0 
\end{aligned}
\end{equation}
and that for each $j \in\{1,\ldots,n\}$, 
\begin{equation}\label{eq:check02}
\begin{aligned}
\|\alpha_j  \mathbf{s}(h_1,h_2)\|_2^2
&=\int_0^T \alpha_j^2h_1^2(t)dt+\int_0^T \alpha_j^2h_2^2(t)dt\\
&=\|\alpha_j h_1\|_2^2 +\|\alpha_jh_2\|_2^2.
\end{aligned}
\end{equation}
Hence, from   \eqref{eq:check01} and \eqref{eq:check02}, we see that 
$\mathcal A\mathbf{s}(h_1,h_2) 
=\{\alpha_1\mathbf{s}(h_1,h_2),\ldots,\alpha_n\mathbf{s}(h_1,h_2)\}$ 
is an orthogonal set of functions in $L_2[0,T]$  and that the PWZ stochastic  
integrals
\[
\langle{\alpha_j,\mathcal Z_{\mathbf{s}(h_1,h_2)}(x,\cdot)}\rangle
= \langle{\alpha_j \mathbf{s}(h_1,h_2),x }\rangle,\,\,
j\in\{1,\ldots,n\}
\]
form a set of independent Gaussian random variables on $C_0[0,T]$.

\par
Given an orthogonal set  $\mathcal A$ of nonzero functions  in $L_2[0,T]$, 
let $\mathfrak{A}^{(2)}$ 
be the  linear  space of all functionals  $F:C_0[0,T]\to\mathbb C$  of the form
\eqref{eq:cylinder-x} for s-a.e. $x\in C_0[0,T]$, where 
$f:\mathbb R^n  \to \mathbb C$ is in  the Hilbert space  $L_2(\mathbb R^{n})$.
Note that $F\in \mathfrak{A}^{(2)}$  implies that $F$ is scale-invariant 
measurable.

 \begin{remark}
In this paper the Hilbert space  $L_2(\mathbb R^n)$ 
has  the  complex scalar field $\mathbb C$.
Thus the inner product $(\cdot, \cdot)_{L_2(\mathbb R^n)}$ on $L_2(\mathbb R^n)$ 
is given by
\[
(f_1, f_2)_{L_2(\mathbb R^n)}=\int_{\mathbb R^n} f_1(\vec{u})\overline{f_2(\vec{u})}d\vec{u}
\] 
for  $f_1, f_2 \in L_2(\mathbb R^n)$.
\end{remark}

\par 
We define a sesquilinear form 
$\langle\!\langle{\cdot,\cdot}\rangle\!\rangle_{\mathfrak{A}^{(2)}}$
on the linear space $\mathfrak{A}^{(2)}$ as follows:
for $F_1$ and $F_2$ in $\mathfrak{A}^{(2)}$, let 
\begin{equation}\label{sesquilinear}
\langle\!\langle{F_1,F_2}\rangle\!\rangle_{\mathfrak{A}^{(2)}}
:=(f_1, f_2)_{L_2(\mathbb R^n)}=\int_{\mathbb{R}^n} f_1(\vec{u})\overline{f_2(\vec{u})}d\vec{u},
\end{equation}
where $f_j$, $j\in\{1,2\}$, is the corresponding function of $F_j$
by equation \eqref{eq:cylinder-x}.
Then one can  see that the sesquilinear form  \eqref{sesquilinear}
is well-defined and it is an inner product on  the linear 
space $\mathfrak{A}^{(2)}$.  Also, one can see that the correspondence \eqref{eq:cylinder-x}, 
$F \mapsto f$, between $\mathfrak{A}^{(2)}$ and $L_2(\mathbb R^n)$ is one-to-one and onto.
Thus $(\mathfrak{A}^{(2)}, \|\cdot\|_{\mathfrak{A}^{(2)}})$ 
forms a complex Hilbert space, where $\|\cdot\|_{\mathfrak{A}^{(2)}}
:=\sqrt{\langle\!\langle{\cdot, \cdot}\rangle\!\rangle_{\mathfrak{A}^{(2)}}}$.

\par
Note that $(\mathfrak{A}^{(2)}, \|\cdot\|_{\mathfrak{A}^{(2)}})$
is a rich class of functionals on $C_0[0,T]$, 
because it contains  many unbounded functionals.
It is well known that  the  Schwartz class $\mathcal{S}(\mathbb{R}^n)$ 
is a dense subspace of $L_2(\mathbb R^n)$. The  functionals $F \in \mathfrak{A}^{(2)}$ 
whose  corresponding function  $f$ by equation \eqref{eq:cylinder-x} 
is of class $\mathcal{S}(\mathbb{R}^n)$  are of interest 
in Feynman integration theory and quantum mechanics.

\par
Throughout this and next  sections,  for convenience, we use the following 
notation: given an orthogonal set $\mathcal A=\{\alpha_1,\ldots,\alpha_n\}$ of nonzero
functions in  $L_2[0,T]$, $f\in L_2(\mathbb R^n)$, $q\in \mathbb R\setminus\{0\}$ 
and $h\in \mathcal O_{\infty}(\mathcal A)$, let
\begin{equation}\label{eq:int-form-fn}
\begin{aligned}
\psi_{f,\mathcal A h}^q(\vec r)
&\equiv \psi_{f,\mathcal A h}^q( r_1,\ldots,r_n)\\
:&=\Bigg(\prod_{j=1}^n \frac{-iq}{2\pi\|\alpha_jh\|_2^2}\Bigg)^{1/2}
\int_{\mathbb R^n} f(\vec u)\exp\Bigg\{\frac{iq}{2}\sum_{j=1}^n
\frac{(u_j-r_j)^2}{\|\alpha_jh\|_2^2}\Bigg\}d\vec u.
\end{aligned}
\end{equation}

\par
In \cite{HPS95}, Huffman, Park, and Skoug established the existence
of the $L_2$ analytic $\mathcal Z_1$-FFT for cylinder 
functionals having the form \eqref{eq:cylinder-x}. 
The following theorem is a restatement of Theorem 2.2 in \cite{HPS95}.

\begin{theorem}\label{thm:admix-trans}
Let $F\in \mathfrak{A}^{(2)}$ be given by equation \eqref{eq:cylinder-x}
and let $h$ be an element of $\mathcal O_{\infty}(\mathcal A)$. 
Then 
\begin{itemize} 
\item[(i)]
for all nonzero real $q$,
the $L_2$  analytic $\mathcal Z_h$-FFT   of $F$, $T_{q,h}^{(2)}(F)$, 
exists, belongs to $\mathfrak{A}^{(2)}$ and is given by the formula
\[
T_{q,h}^{(2)}(F)(y)
=\psi_{f,\mathcal A h}^q(\langle{\alpha_1h,x}\rangle,
\ldots,\langle{\alpha_nh,x}\rangle)
\]
for s-a.e. $y\in C_0[0,T]$, where $\psi_{f,\mathcal A h}^q$ 
is given by \eqref{eq:int-form-fn}; and
\item[(ii)]
for all nonzero real $q$, 
\[
T_{-q,h}^{(2)}(T_{q,h}^{(2)}(F))\approx F.
\]
In other words, the $L_2$ analytic  $\mathcal Z_h$-FFT, $T_{q,h}^{(2)}$ 
has the inverse transform 
\[
\{T_{q,h}^{(2)}\}^{-1}=T_{-q,h}^{(2)}.
\]
\end{itemize}
\end{theorem}

\begin{theorem}\label{p:multi}
Let $F\in \mathfrak{A}^{(2)}$ be given by equation \eqref{eq:cylinder-x},
and  let $h_1$ and $h_2$ be  elements of $\mathcal O_{\infty}(\mathcal A)$. 
Then for all nonzero real $q$,
\begin{equation}\label{two-single}
T_{q,h_2}^{(2)}(T_{q,h_1}^{(2)}(F))=T_{q,\mathbf{s}(h_1,h_2)}^{(2)}(F)(y)
\end{equation}
for s-a.e. $y\in C_0[0,T]$.
\end{theorem}
\begin{proof}
In view of Theorem \ref{thm:admix-trans}, the two  FFTs
in equation  \eqref{two-single} exist. Thus equality is what needs to be shown. 
Hence, to establish equation \eqref{two-single}, it will suffice  to show that 
for each $\lambda>0$,
\[
T_{\lambda,h_2}(T_{\lambda,h_1}(F))(y)=T_{\lambda,\mathbf{s}(h_1,h_2)}(F)(y)
\]
for s-a.e. $y\in C_0[0,T]$. But, using  equation \eqref{R001}, 
for each $\lambda>0$ and s-a.e. $y\in C_0[0,T]$, we obtain
\[
\begin{aligned}
&T_{\lambda,h_2}(T_{\lambda,h_1}(F))(y)\\
&=\int_{C_0^2[0,T]}
F (y+\lambda^{-1/2}\mathcal Z_{h_1}(x_1,\cdot)
+\lambda^{-1/2}\mathcal Z_{h_2}(x_2,\cdot) )dm^2(x_1,x_2)\\
&=\int_{C_0^2[0,T]}F (y+\mathcal Z_{h_1/\sqrt{\lambda}}(x_1,\cdot)
+\mathcal Z_{h_2/\sqrt{\lambda}}(x_2,\cdot) )dm^2(x_1,x_2)\\
&=\int_{C_0[0,T]}F 
(y+\mathcal Z_{\mathbf{s}(h_1/\sqrt{\lambda},h_2/\sqrt{\lambda})}(x,\cdot) )dm (x)\\
&=\int_{C_0[0,T]}F (y+\lambda^{-1/2}\mathcal Z_{\mathbf{s}(h_1,h_2)}(x,\cdot) )dm (x)\\
&=T_{\lambda,\mathbf{s}(h_1,h_2)}(F)(y).
\end{aligned}
\]
Thus,  by the definition of the FFT,  we obtain the desired result.
\end{proof}

\par
Using mathematical induction and equation \eqref{R007}, 
we obtain the following corollary.

\begin{corollary}
Let $F\in \mathfrak{A}^{(2)}$ be given by equation \eqref{eq:cylinder-x},
and  let $\mathcal H=(h_1,\ldots,h_n)$ be a finite sequence in  $\mathcal O_{\infty}(\mathcal A)$. 
Then for all nonzero real $q$,
\begin{equation}\label{R004TT}
T_{q,h_n}^{(2)}(\cdots(T_{q,h_1}^{(2)}(F))\cdots)=T_{q,\mathbf{s}(\mathcal H)}^{(2)}(F)(y)
\end{equation}
for s-a.e. $y\in C_0[0,T]$.
\end{corollary}

\par
Applying equation \eqref{R002}, we also obtain the following corollary.

\begin{corollary}
Let $F\in \mathfrak{A}^{(2)}$ be given by equation \eqref{eq:cylinder-x},
and  let $\mathcal H_1$ and $\mathcal H_2$ be  finite sequences in  
$\mathcal O_{\infty}(\mathcal A)$. Then for all nonzero real $q$,
\begin{equation}\label{R003T}
T_{q,\mathbf{s}(\mathcal H_2)}^{(2)}(T_{q,\mathbf{s}(\mathcal H_1)}^{(2)}(F))(y)
=T_{q,\mathbf{s}(\mathcal H_1\wedge \mathcal H_2)}^{(2)}(F)(y)
\end{equation}
for s-a.e. $y\in C_0[0,T]$.
\end{corollary}

\setcounter{equation}{0}
\section{Algebraic structures of  $\mathcal Z_h$-Fourier--Feynman transforms}

\par
Given $q\in\mathbb R\setminus\{0\}$, let
\[
\mathsf{T}_{q,\mathcal O_{\infty}(\mathcal A)}
\equiv \mathsf{T}_{q,\mathcal O_{\infty}(\mathcal A)}[\mathfrak{A}^{(2)}]
:=\{T_{q,h}^{(2)}: h \in \mathcal O_{\infty}(\mathcal A)\cup\{0\}\}
\]
denote the class of   $L_2$ analytic $\mathcal Z_h$-FFTs acting on $\mathfrak{A}^{(2)}$.
In the case that $h\equiv 0$, i.e., $\|h\|_2=0$, it follows that  $T_{q,h}^{(2)}\equiv T_{q,0}^{(2)}$
is the identity transform  for all $q\in\mathbb R$.
For notational convenience, let $\mathbf{s}(h)\equiv h$ for $h\in L_{\infty}[0,T]$.

\par
By Theorems \ref{thm:admix-trans} and  \ref{p:multi}, we see that
for all $h_1,h_2 \in \mathcal O_{\infty}(\mathcal A)\cup\{0\}$ and all
$F\in\mathfrak{A}^{(2)}$,
\[
(T_{q,h_2}^{(2)}\circ  T_{q,h_1}^{(2)})(F)
 \equiv  (T_{q,h_2}^{(2)}(T_{q,h_1}^{(2)}(F))
=T_{q,\mathbf{s}(h_1,h_2)}^{(2)}(F)
\]
is in $\mathfrak{A}^{(2)}$. Because 
\[
\mathbf{s}(\mathbf{s}(h_3,h_2),h_1)
=\mathbf{s}(h_3,h_2,h_1)=\mathbf{s}(h_3,\mathbf{s}(h_2,h_1)),
\]
for all $h_1,h_2,h_3 \in \mathcal O_{\infty}(\mathcal A)\cup\{0\}$,
we see that the composition $\circ$ of $L_2$  analytic FFTs 
is associative. Also, because  $\mathbf{s}(h_1,h_2)=\mathbf{s}(h_2,h_1)$, 
we see that 
$(T_{q,h_2}^{(2)}\circ  T_{q,h_1}^{(2)})(F) =(T_{q,h_1}^{(2)}\circ  T_{q,h_2}^{(2)})(F)$,
and clearly $(T_{q,0}^{(2)}\circ  T_{q,h}^{(2)})(F)=T_{q,h}^{(2)}(F)$
for any $h_1,h_2, h\in \mathcal O_{\infty}(\mathcal A)\cup\{0\}$ 
and every $F\in\mathfrak{A}^{(2)}$. Thus, we have the following assertion.

\begin{theorem}
The space $(\mathsf{T}_{q,\mathcal O_{\infty}(\mathcal A)} ,\circ)$
is a commutative monoid. Furthermore, the monoid
$\mathsf{T}_{q,\mathcal O_{\infty}(\mathcal A)} $ acts on the
space  $\mathfrak{A}^{(2)}$ in the sense that 
$(T_{q,h}^{(2)}, F) \mapsto  T_{q,h}^{(2)}(F)$.
\end{theorem}

\par
Next  let ${\rm S}_{\rm f}(\mathcal O_{\infty}(\mathcal A))$ be the set of 
all finite sequences of functions in $\mathcal O_{\infty}(\mathcal A)\cup\{0\}$, 
and let
\[
\mathsf{T}_{q,{\rm S}_{\rm f}(\mathcal O_{\infty}(\mathcal A))}
 \equiv \mathsf{T}_{q,{\rm S}_{\rm f}(\mathcal O_{\infty}(\mathcal A))}[\mathfrak{A}^{(2)}]
:=\{T_{q,\mathbf{s}(\mathcal H)}^{(2)}: \mathcal H \in {\rm S}_{\rm f}(\mathcal O_{\infty}(\mathcal A))\}.
\]

\par
From the fact that  for any 
$\mathcal H \in {\rm S}_{\rm f}(\mathcal O_{\infty}(\mathcal A))$,
$\mathbf{s}(\mathcal H)\in \mathcal O_{\infty}(\mathcal A)\cup\{0\}
\subset {\rm S}_{\rm f}(\mathcal O_{\infty}(\mathcal A))$,  
we know that the classes $\mathsf{T}_{q,\mathcal O_{\infty}(\mathcal A)}$
and $\mathsf{T}_{q,{\rm S}_{\rm f}(\mathcal O_{\infty}(\mathcal A))}$
coincide as sets. However, we will consider another operation, $\barwedge$, 
on $\mathsf T_{q,{\rm S}_{\rm f}(\mathcal O_{\infty}(\mathcal A))}$,
defined as follows: for 
$T_{q,\mathbf{s}(\mathcal H_{1})}^{(2)}$ and $T_{q,\mathbf{s}(\mathcal H_{2})}^{(2)}$
in $\mathsf T_{{\rm S}_{\rm f}(\mathcal O_{\infty}(\mathcal A))}$, let
\[
T_{q,\mathbf{s}(\mathcal H_{1})}^{(2)} \barwedge T_{q,\mathbf{s}(\mathcal H_{2})}^{(2)}
:=T_{q,\mathbf{s}(\mathcal H_{1} \wedge \mathcal H_{2})}^{(2)},
\]
where $\mathcal H_{1} \wedge \mathcal H_{2}$  is defined by \eqref{wedge-H} above.
By equation \eqref{s-wedge-H}, one can see that the operation $\barwedge$ 
is well defined.

\begin{remark}
By equation  \eqref{R003T}, we observe 
\[
T_{q,\mathbf{s}(\mathcal H_{1})}^{(2)} 
\barwedge T_{q,\mathbf{s}(\mathcal H_{2})}^{(2)}
=T_{q,\mathbf{s}(\mathcal H_{1} \wedge \mathcal H_{2})}^{(2)}
=T_{q,\mathbf{s}(\mathcal H_{1})}^{(2)} 
\circ T_{q,\mathbf{s}(\mathcal H_{2})}^{(2)}.
\]
\end{remark}

\begin{theorem}
The space
$(\mathsf{T}_{q,{\rm S}_{\rm f}(\mathcal O_{\infty}(\mathcal A))},\barwedge)$
is a commutative monoid. Furthermore, the monoid
$\mathsf{T}_{q,{\rm S}_{\rm f}(\mathcal O_{\infty}(\mathcal A))}$
acts on the space  $\mathfrak{A}^{(2)}$ in the sense that 
$(T_{q,\mathbf{s}(\mathcal H)}^{(2)}, F) 
\mapsto T_{q,\mathbf{s}(\mathcal H)}^{(2)}(F)$.
\end{theorem}

\begin{remark}
The operation $\barwedge$ is a semigroup action of
$\mathsf{T}_{q,{\rm S}_{\rm f}(\mathcal O_{\infty}(\mathcal A))}$
on $\mathfrak{A}^{(2)}$.
\end{remark}

\par
The sequence space ${\rm S}_{\rm f}(\mathcal O_{\infty}(\mathcal A))$ is a 
monoid under the operation $\wedge$  given by \eqref{wedge-H}. 
Define an equivalence relation $\stackrel{\mathbf{s}}{\sim}$ on
${\rm S}_{\rm f}(\mathcal O_{\infty}(\mathcal A))$ as follows: 
for $\mathcal H_{1}$ and $\mathcal H_{2}$ 
in ${\rm S}_{\rm f}(\mathcal O_{\infty}(\mathcal A))$,
\[
\mathcal H_{1} \stackrel{\mathbf{s}}{\sim} \mathcal H_{2}
\Longleftrightarrow  \mathbf{s}(\mathcal H_{1})=\mathbf{s}(\mathcal H_{2}).
\]
Also, let
\[
{\rm S}_{\rm f}^{\sim}
\equiv  {\rm S}_{\rm f}(\mathcal O_{\infty}(\mathcal A))/\stackrel{\mathbf{s}}{\sim}
\,:= \{ [\mathcal H]_{\mathbf{s}} : 
\mathcal H\in  {\rm S}_{\rm f}(\mathcal O_{\infty}(\mathcal A)) \}
\]
be the quotient set of ${\rm S}_{\rm f}(\mathcal O_{\infty}(\mathcal A))$ by
$\stackrel{\mathbf{s}}{\sim}$. Then from  \eqref{eq:s-def} and \eqref{eq:pre-eq}, 
we  see that  ${\rm S}_{\rm f}^{\sim}$ is the quotient
monoid under the operation $\wedge$ on ${\rm S}_{\rm f}^{\sim}$, given by
\begin{equation}\label{eq:quotientop-s}
 [\mathcal H_{1}]_{\mathbf{s}} \wedge [\mathcal H_{2}]_{\mathbf{s}}
:= [\mathcal H_{1} \wedge \mathcal H_{2}]_{\mathbf{s}}.
\end{equation}

\par
Define a   relation on 
$\mathsf{T}_{q,{\rm S}_{\rm f}(\mathcal O_{\infty}(\mathcal A))}$
as follows: for
$T_{q,\mathbf{s}(\mathcal H_{1})}^{(2)}$ and $ T_{q,\mathbf{s}(\mathcal H_{2})}^{(2)}$  in
$\mathsf{T}_{q,{\rm S}_{\rm f}(\mathcal O_{\infty}(\mathcal A))}$,
\[
T_{q,\mathbf{s}(\mathcal H_{1})}^{(2)}
\stackrel{\text{\rm t}}{\sim} T_{q,\mathbf{s}(\mathcal H_{2})}^{(2)}
\Longleftrightarrow
\mathcal H_{1} \stackrel{\mathbf{s}}{\sim} \mathcal H_{2}.
\]
From   \eqref{R004TT}   and   \eqref{eq:pre-eq}, we  see that
for every $(h_1,\ldots,h_n)\in {\rm S}_{\rm f}(\mathcal O_{\infty}(\mathcal A))$  
and   any permutation $\pi$ of $\{1,\ldots,n\}$,
\[
T_{q,\mathbf{s}(h_1,\ldots,h_n)}^{(2)}(F)
= T_{q,\mathbf{s}(h_{\pi(1)},\ldots,h_{\pi(n)})}^{(2)}(F)
\]
for all $F$  in $\mathfrak{A}^{(2)}$. Thus, the relation $\stackrel{\text{\rm t}}{\sim}$ 
is a well-defined equivalence relation, and so we can obtain the  quotient monoid
\[
\begin{aligned}
{\mathsf{T}}_{q,{\rm S}_{\rm f}(\mathcal O_{\infty}(\mathcal A))}^{\sim} 
&\equiv \mathsf{T}_{q,{\rm S}_{\rm f}(\mathcal O_{\infty}(\mathcal A))}/\stackrel{\text{\rm t}}{\sim}\\
:&= \{[T_{q,\mathbf{s}(\mathcal H)}^{(2)}]_{\text{\rm t}}:T_{q,\mathbf{s}(\mathcal H)}^{(2)}
\in  \mathsf{T}_{q,{\rm S}_{\rm f}(\mathcal O_{\infty}(\mathcal A))}\}
\end{aligned}
\]
with the operation  $\barwedge$   given by
\begin{equation}\label{eq:quotientop-t}
[T_{q,\mathbf{s}(\mathcal H_{1})}^{(2)}]_{\text{\rm t}}
\barwedge [T_{q,\mathbf{s}(\mathcal H_{2})}^{(2)}]_{\text{\rm t}}
:=[T_{q,\mathbf{s}(\mathcal H_{1} \wedge  \mathcal H_{2} )}^{(2)}]_{\text{\rm t}}.
\end{equation}

\begin{theorem}
The map 
$\Xi:({\mathsf{T}}_{q,{\rm S}_{\rm f}(\mathcal O_{\infty}(\mathcal A))}^{\sim}, \barwedge )
\to ({\rm S}_{\rm f}^{\sim}, \wedge )$ given by
\begin{equation}\label{eq:xi=map}
\Xi([T_{q,\mathbf{s}(\mathcal H)}^{(2)}]_{\text{\rm t}})=[\mathcal H]_{\mathbf{s}}.
\end{equation}
is a monoid isomorphism.
\end{theorem}
\begin{proof}
It follows from \eqref{eq:quotientop-t} and \eqref{eq:quotientop-s}  
that
\[
\begin{aligned}
 \Xi([T_{q,\mathbf{s}(\mathcal H_1)}^{(2)}]_{\text{\rm t}}
\barwedge [T_{q,\mathbf{s}(\mathcal H)}^{(2)}]_{\text{\rm t}})
&=\Xi( [T_{q,\mathbf{s}(\mathcal H_{1} \wedge \mathcal H_{2})}^{(2)}]_{\text{\rm t}})\\
&=[ \mathcal H_{1} \wedge \mathcal H_{2} ]_{\mathbf{s}}\\
&=[ \mathcal H_{1} ]_{\mathbf{s}}  \wedge [\mathcal  H_{2}]_{\mathbf{s}} \\
&=\Xi([T_{q,\mathbf{s}(\mathcal H_{1})}^{(2)}]_{\text{\rm t}}) \barwedge
\Xi([T_{q,\mathbf{s}(\mathcal H_{2})}^{(2)}]_{\text{\rm t}}).
\end{aligned}
\]
Clearly, the map given by equation \eqref{eq:xi=map} is   bijective.
\end{proof}

\section{Free group 
$\mathsf{F}({\mathsf{T}}_{q,{\rm S}_{\rm f}(\mathcal O_{\infty}^n(\mathcal A))}^{\sim})$}

\par
In this section, we describe a transformation group that is the free group  
generated by ${\mathsf{T}}_{q,{\rm S}_{\rm f}(\mathcal O_{\infty}(\mathcal A))}^{\sim}$.

\par
Given an orthogonal set $\mathcal A$ of functions in $L_2[0,T]$,
let $\mathcal O_{\infty}^n(\mathcal A)$ be  the class of all nonzero 
functions $h$ in $L_{\infty}[0,T]$,  such that $\mathcal{A}h$ is orthonormal 
in $L_2[0,T]$. For a detailed example for the class  $\mathcal O_{\infty}^n(\mathcal A)$, 
see \cite[Example 2.2]{CC13}.

\par
The following lemma is due to Cameron and Storvick in \cite[Lemma H]{CS76-1}.

\begin{lemma} \label{lem:L2}
For $f \in L_2(\mathbb{R}^n)$ and $h\in\mathcal O_{\infty}^n(\mathcal A)$,
let $\psi_{f,\mathcal A h}^q(\vec r)$ be given by equation \eqref{eq:int-form-fn}. 
Then $\psi_{f,\mathcal A h}^q\in L_2(\mathbb{R}^n)$. The integral in the right 
side of \eqref{eq:int-form-fn}  is to be interpreted as an $L_2$-limiting integral
in the sense that
\[
\begin{aligned}
&\lim_{\delta \to \infty} \int_{\mathbb{R}^n} \Bigg|\bigg(\frac{-iq}{2\pi}\bigg)^{n/2}
\int_{-\delta}^\delta \stackrel{(n)}\cdots \int_{-\delta}^\delta
f(\vec u)\exp\Bigg\{\frac{iq}{2}\sum_{j=1}^n(u_j-r_j)^2\Bigg\}d\vec{u}\\
&\qquad\qquad\qquad\qquad\qquad\qquad\qquad\qquad\qquad\qquad\qquad
- \psi_{f,\mathcal A h}^q(\vec r)\Bigg|^2d\vec{r}=0.
\end{aligned}
\]
In this case, we have $\|\psi_{f,\mathcal A h}^q \|_2=\|f\|_2$.
\end{lemma}

\par
In view of Theorem \ref{thm:admix-trans} and  Lemma \ref{lem:L2}, 
we obtain the following theorem.

\begin{theorem}\label{thm:Hilbert-iso}
Let  $q\in\mathbb R\setminus\{0\}$ and let $h\in\mathcal O_{\infty}^n(\mathcal A)$.
Then  $L_2$  analytic $\mathcal Z_h$-FFT,
\[
T_{q,h}^{(2)}: \mathfrak{A}^{(2)} \to  \mathfrak{A}^{(2)}
\]
is a linear operator  isomorphism. That is,
$\|F\|_{\mathfrak{A}^{(2)}}=\|T_{q,h}^{(2)}(F)\|_{\mathfrak{A}^{(2)}}$
for all $F\in \mathfrak{A}^{(2)}$. Thus we have $\|T_{q,h}^{(2)} \|_{\rm o}=1$, 
where $\|\cdot\|_{\rm o}$ denotes the operator norm.
\end{theorem}

\par
For any nonzero real   $q$, let 
${{\mathsf{T}}_{q,{\rm S}_{\rm f}(\mathcal O_{\infty}^n(\mathcal A))}^{\sim}}^{\!\!\!\!*}
:=  {\mathsf{T}}_{q,{\rm S}_{\rm f}(\mathcal O_{\infty}^n(\mathcal A))}^{\sim}
\setminus\{[T_{q,0}^{(2)}]_{\mathbf{s}}\}$.
Given $q\in\mathbb R\setminus\{0\}$, define a map 
\[
\mathcal W:{{\mathsf{T}}_{q,{\rm S}_{\rm f}(\mathcal O_{\infty}^n(\mathcal A))}^{\sim}}^{\!\!\!\!*}
\longrightarrow 
{{\mathsf{T}}_{-q,{\rm S}_{\rm f}(\mathcal O_{\infty}^n(\mathcal A))}^{\sim}}^{\!\!\!\!*}
\]
by $\mathcal W([T_{q,\mathbf{s}(\mathcal H)}^{(2)}]_{\mathrm{t}})
=[T_{-q,\mathbf{s}(\mathcal H)}^{(2)}]_{\mathrm{t}}$.
Then, $\mathcal W$ is one-to-one correspondence. Thus, by the usual argument 
in the free group theory, one can obtain  the group
$\mathsf{F}({\mathsf{T}}_{q,{\rm S}_{\rm f}(\mathcal O_{\infty}^n(\mathcal A))}^{\sim})$ 
freely generated by  
${{\mathsf{T}}_{q,{\rm S}_{\rm f}(\mathcal O_{\infty}^n(\mathcal A))}^{\sim}}^{\!\!\!\!*}$.

\par
Note that 
\[
[T_{q,\mathbf{s}(\mathcal H_{1})}^{(2)}]_{\mathrm{t}} 
\barwedge [T_{q,\mathbf{s}(\mathcal H_{2})}^{(2)}]_{\mathrm{t}}
=[T_{q,\mathbf{s}(\mathcal H_{1} \wedge \mathcal H_{2})}^{(2)}]_{\mathrm{t}}
=[T_{q,\mathbf{s}(\mathcal H_{1})}^{(2)} 
\circ T_{\mathbf{s}(\mathcal H_{2})}^{(2)}]_{\mathrm{t}}.
\]
by equation  \eqref{R003T}. Given two transforms $\mathcal T_1$ and $\mathcal T_2$ in 
$\mathsf{F}({\mathsf{T}}_{q,{\rm S}_{\rm f}(\mathcal O_{\infty}^n(\mathcal A))}^{\sim})$,
let the group operation between $\mathcal T_1$ and $\mathcal T_2$ be given by
\[
(\mathcal T_1 \circ \mathcal T_2)(F)
\equiv  \mathcal T_1( \mathcal T_2(F)),\,\,\,  F\in \mathfrak{A}^{(2)}.
\]

\par
For an element  $\mathcal T$ of 
$\mathsf{F}({\mathsf{T}}_{q,{\rm S}_{\rm f}(\mathcal O_{\infty}^n(\mathcal A))}^{\sim})$,
let $\mathsf l_{w}(\mathcal T)$ denote the length of the word $\mathcal T$.
Given  $\mathcal T\in \mathsf{F}({\mathsf{T}}_{q,{\rm S}_{\rm f}
(\mathcal O_{\infty}^n(\mathcal A))}^{\sim})$, assume that $\mathcal T$ is  not the empty 
word (i.e., it is not the identity transform $[T_{q,0}^{(2)}]_{\mathrm{t}}$).
If $\mathsf l_{w}(\mathcal T)=1$, then  $\mathcal T$ is an element of the set
\[
{{\mathsf{T}}_{q,{\rm S}_{\rm f}(\mathcal O_{\infty}^n(\mathcal A))}^{\sim}}^{\!\!\!\!*}
\,\,\,\dot\cup \,\,\,
{{\mathsf{T}}_{-q,{\rm S}_{\rm f}(\mathcal O_{\infty}^n(\mathcal A))}^{\sim}}^{\!\!\!\!*}\,\,.
\]
 Alternatively, if $\mathsf{l}_{w}(\mathcal T)>1$, then $\mathcal T$ cannot be expressed 
as   (an equivalence class of) a single  FFT by the concept of the reduced word 
in the free group theory. But, in view of  the assertion (ii) of Theorem \ref{thm:admix-trans} 
and Lemma \ref{lem:L2}, we see that for any  
$\mathcal T\in \mathsf{F}({\mathsf{T}}_{q,{\rm S}_{\rm f}(\mathcal O_{\infty}^n(\mathcal A))}^{\sim})$, 
$\mathcal T$ is a linear operator  isomorphism  from  $\mathfrak{A}^{(2)}$ into $\mathfrak{A}^{(2)}$.

\par
On the other hand, we consider  other algebraic structure 
of transforms as follows: given  $h\in L_{\infty}[0,T]$, let
$T_{0,h}^{(2)}$ denote the identity transform, i.e., $T_{0,h}^{(2)}(F)=F$, 
and let
\[
 \mathsf{T}_{\mathbb R,h} 
=\big\{T_{q,h}^{(2)}: q\in  \mathbb R \big\}.
\] 
Now, by using \cite[equation (2.14)]{HSS2001} (clearly, 
equations (2.10) and (2.14) in \cite{HSS2001} hold for $L_2$  analytic 
 FFTs $T_{q,h}^{(2)}$), we obtain   that for all $q_1,q_2\in\mathbb R$
with $q_1+q_2\ne0$ and all $F\in \mathfrak{A}^{(2)}$,
\[
 (T_{q_2,h}^{(2)}\circ  T_{q_1,h}^{(2)} )(F)
\equiv  T_{q_2,h}^{(2)} ( T_{q_1,h}^{(2)}(F) )
\approx T_{\frac{q_1q_2}{q_1+q_2},h}^{(2)}(F),
\]
and that
\[
 (T_{q_2,h}^{(2)}\circ  T_{q_1,h}^{(2)} )(F)
\approx  (T_{q_1,h}^{(2)}\circ  T_{q_2,h}^{(2)} )(F).
\]
If  $q_1+q_2=0$, then by Theorem \ref{thm:admix-trans},
we obtain  $T_{q_2,h}^{(2)} =T_{-q_1,h}^{(2)}
=\{T_{q_1,h}^{(2)}\}^{-1}$.

\begin{theorem}
For each $h\in L_{\infty}[0,T]$, the space $(\mathsf{T}_{\mathbb R,h},\circ)$
forms a  commutative group. Clearly, $\mathsf{T}_{\mathbb R,h}$ with $\|h\|_2=0$ 
is a trivial group.
\end{theorem}

\par
We note that  given $h\in L_{\infty}[0,T]$ with $\|h\|_2>0$, 
the transformation group $\mathsf{T}_{\mathbb R,h}$ is   
the free group 
with the free basis $\{T_{q,h}^{(2)}: q>0\}$.
Because
\[
|\{T_{q,h}^{(2)}: q>0\}|
=\aleph_1
=\aleph_0^{\aleph_0}=|{\rm S}_{\rm f}|
\equiv  |{{\mathsf{T}}_{q,{\rm S}_{\rm f}
(\mathcal O_{\infty}^n(\mathcal A))}^{\sim}}^{\!\!\!\!*}|,
\]
the free groups 
$\mathsf{F}({\mathsf{T}}_{q,{\rm S}_{\rm f}(\mathcal O_{\infty}^n(\mathcal A))}^{\sim})$ 
and $\mathsf{T}_{\mathbb R,h}$ have the same rank. Thus, by the concept of rank
of free groups \cite[Proposition 2.1.4, p.47]{Robinson}, we conclude that
\[
\mathsf{F}({\mathsf{T}}_{q,{\rm S}_{\rm f}(\mathcal O_{\infty}^n(\mathcal A))}^{\sim})
\stackrel{\iota}{\cong}\mathsf{T}_{\mathbb R,h},
\]
where $\iota$ is a group isomorphism.

\section*{Acknowledgments}
This research was supported by the research fund of Dankook University in 2017.


\end{document}